\documentclass[aap,preprint]{imsart}
\RequirePackage{amsthm,amsmath,amsfonts,amssymb}
\RequirePackage[colorlinks,citecolor=blue,urlcolor=blue]{hyperref}
\usepackage{mathtools}
\usepackage{xparse}

\usepackage{enumitem}
\usepackage{bbm}

\usepackage[numeric,initials,nobysame]{amsrefs}
\usepackage{verbatim}

\makeatletter
\@namedef{subjclassname@2020}{%
  \textup{2020} Mathematics Subject Classification}
\makeatother

\newtheorem{thm}{Theorem}
\newtheorem{mainthm}{Theorem}
\newtheorem{cor}[thm]{Corollary}
\newtheorem{lem}[thm]{Lemma}
\newtheorem{prop}[thm]{Proposition}

\newtheorem{claim}[thm]{Claim}
\theoremstyle{definition}

\newtheorem{rem}[thm]{Remark}
\newtheorem{ex}[thm]{Example}

\newtheorem{obs}[thm]{Observation}

\numberwithin{thm}{section}

\newcommand{\cE}{\ensuremath{\mathcal E}}
\newcommand{\cF}{\ensuremath{\mathcal F}}

\newcommand{\cL}{\ensuremath{\mathcal L}}

\newcommand{\cR}{\ensuremath{\mathcal R}}

\newcommand{\bbE}{{\ensuremath{\mathbb E}} }

\newcommand{\bbP}{{\ensuremath{\mathbb P}} }

\newcommand{\bbR}{{\ensuremath{\mathbb R}} }

\newcommand{\bbT}{{\ensuremath{\mathbb T}} }

\newcommand{\bbZ}{{\ensuremath{\mathbb Z}} }

\let\a=\alpha \let\b=\beta   \let\d=\delta  
\let\f=\varphi \let\g=\gamma \let\h=\eta      \let\l=\lambda
\let\m=\mu   \let\n=\nu   \let\o=\omega      
\let\r=\rho  \let\s=\sigma \let\t=\tau

\let\O=\Omega      \let\X=\Xi

\newcommand{\tc}{\thinspace |\thinspace}
\newcommand{\var}{\operatorname{Var}}

\newcommand{\1}{{\ensuremath{\mathbbm{1}}} }
\newcommand{\ent}{{\ensuremath{\mathrm{Ent}}} }

\renewcommand{\le}{\leqslant}
\renewcommand{\ge}{\geqslant}
\renewcommand{\to}{\rightarrow}

\newcommand{\tcov}{{\ensuremath{T^{\rm rw}_{\rm cov}}} }

\newcommand{\scov}{{\ensuremath{\s_{\rm cov}}} }
\newcommand{\tmeet}{{\ensuremath{T_{\rm meet}^{\mathrm{\scriptstyle{rw}}}}} }
\DeclareDocumentCommand \tmix { o } {%
  \IfNoValueTF {#1} {{\ensuremath{T_{\rm mix}}} }{{\ensuremath{T_{\rm mix}^{\mathrm{\scriptstyle{#1}}}}}}%
}
\DeclareDocumentCommand \tinf { o } {%
  \IfNoValueTF {#1} {{\ensuremath{T_{\infty}}} }{{\ensuremath{T_{\infty}^{\mathrm{\scriptstyle{#1}}}}}}%
}
\DeclareDocumentCommand \tq { o } {%
  \IfNoValueTF {#1} {{\ensuremath{T_{q}}} }{{\ensuremath{T_{q}^{\mathrm{\scriptstyle{#1}}}}}}%
}
\DeclareDocumentCommand \trel { o } {%
  \IfNoValueTF {#1} {{\ensuremath{T_{\rm rel}}} }{{\ensuremath{T_{\rm rel}^{\mathrm{\scriptstyle{#1}}}}}}%
}
\DeclareDocumentCommand \csob { o } {%
  \IfNoValueTF {#1} {{\ensuremath{\a^{-1}}} }{{\ensuremath{\left(\a^{\mathrm{\scriptstyle{#1}}}\right)^{-1}}}}%
}
\DeclareDocumentCommand \ttwo { o } {%
  \IfNoValueTF {#1} {{\ensuremath{T_{2}}} }{{\ensuremath{T_{2}^{\mathrm{\scriptstyle{#1}}}}}}%
}
\DeclareDocumentCommand \cD { o o } {%
  \IfNoValueTF {#1} {{\ensuremath{\mathcal{D}}} }%
  {\IfNoValueTF {#2} {{\ensuremath{\mathcal{D}^{\mathrm{\scriptstyle{#1}}}}}}%
  {{\ensuremath{\mathcal{D}^{\mathrm{\scriptstyle{#1}}}_{#2}}}}}}

\usepackage{xcolor}

\begin{document}
\begin{frontmatter}
\title{Coalescing and branching simple symmetric exclusion process}
\runtitle{Coalescing and branching SEP}

\begin{aug}
\author[A]{\fnms{Ivailo} \snm{Hartarsky}\ead[label=e1,mark]{hartarsky@ceremade.dauphine.fr}},
\author[B]{\fnms{Fabio} \snm{Martinelli}\ead[label=e2,mark]{martin@mat.uniroma3.it}}
\and
\author[A]{\fnms{Cristina} \snm{Toninelli}\ead[label=e3,mark]{toninelli@ceremade.dauphine.fr}}
\address[A]{CEREMADE, CNRS, Universit\'e Paris-Dauphine, PSL University, \printead{e1}, \printead{e3}}

\address[B]{Dipartimento di Matematica e Fisica, Universit\`a Roma Tre,
\printead{e2}}

\end{aug}

\begin{abstract}
Motivated by kinetically constrained interacting particle systems
(KCM), we consider a reversible coalescing and branching simple
exclusion process on a general finite graph $G=(V,E)$ dual to the
biased voter model on $G$. Our main goal are tight bounds on its
logarithmic Sobolev constant and relaxation time, with particular focus on the delicate
slightly supercritical regime in which the equilibrium density of
particles tends to zero as $|V|\to \infty$. Our results allow us to
recover very directly and improve to $\ell^p$-mixing, $p\ge 2$, and to
more general graphs, the mixing time results of Pillai and Smith for
the Fredrickson-Andersen one spin facilitated (FA-$1$f) KCM on the discrete $d$-dimensional
torus. In view of applications to the more complex FA-$j$f KCM, $j>1$, we also extend part of the analysis to an analogous process with a more general product state space.  
\end{abstract}



\end{frontmatter}

\section{Introduction}
In this work we study a coalescing and branching simple symmetric
exclusion process ({CBSEP}) on a general finite graph $G=(V,E)$. The model was first
introduced by Schwartz \cite{Schwartz77} in 1977 (also see Harris
\cite{Harris76}) as follows. Consider a system of particles performing independent continuous time random
walks on the vertex set of a (finite or infinite) graph $G$ by jumping along each edge with rate 1, which coalesce when
they meet (a particle jumping on top of another one is destroyed) and
which branch with rate $\b>0$ by creating an additional particle at a empty neighbouring vertex. The process is readily seen to be reversible
w.r.t.\ the Bernoulli($p$)-product measure with $p=\frac{\b}{(1+\b)}$. Initially the model was introduced in order to study the biased voter model \cite{Schwartz77} (also known as Williams-Bjerknes tumour growth model \cite{Williams72}), which turns out to be its dual additive interacting particle system \cite{Griffeath79}.\footnote{In fact, biased voter and Williams-Bjerknes models slightly differ on non-regular graphs. For such graphs {CBSEP} is the additive dual of the former.} A further duality in between the two processes in the Sudbury--Lloyd sense \cite{Sudbury97} has been established since then, which shows that the law of CBSEP at a fixed time can be obtained as a $p$-thinning of the biased voter model (see \cite{Swart13}*{Exercise 3.6}). When $\b=0$ this model reduces to coalescing random walks, additive dual to the standard voter model, which have both been extensively studied (see e.g.\ \cites{Liggett99,Liggett05}).

When the graph is the $d$-dimensional hypercubic lattice,
the first results were obtained by Bramson and Griffeath
\cites{Bramson81,Bramson80}. In particular, they showed that the law
of CBSEP converges weakly to its unique invariant measure starting
from any non-empty set of particles and for any dimension
$d$. Moreover, building on their work, Durrett and Griffeath
\cite{Durrett82} proved a shape theorem for this process, which easily
implies that CBSEP on the discrete torus of side length $L\to\infty$
exhibits mixing time cutoff (but without any control on the critical window). In
the case of the regular tree a complete convergence result is due to
Louidor, Tessler, and Vandenberg-Rodes \cite{Louidor14}. In the particular
setting of
$\bbZ$ a key observation is that the rightmost (or leftmost) particle performs a
biased random walk with explicit constant drift (see Griffeath
\cite{Griffeath79}). For more advanced results see e.g.\ the work by Sun and Swart \cite{Sun08}.

While the main focus of the above-mentioned works was the long-time
behavior of the process on \emph{infinite} graphs, our interest will concentrate instead on
the mixing time for \emph{finite} graphs. We determine the logarithmic Sobolev
constant and relaxation time of the model quite precisely on a wide spectrum of relatively
sparse finite graphs and for values of the branching rate $\b$ which
are o(1) as $|V|\to\infty$ (see Theorem \ref{mainthm 1} and Corollary
\ref{cor: 1}). For instance,
our results imply that for transitive bounded degree graphs the
inverse of the logarithmic Sobolev constant and relaxation time when $\b=1/|V|$ are, up to a
logarithmic correction, equal to the cover time of the graph. We will
then use these results to strengthen and extend the findings of Pillai
and Smith \cites{Pillai17,Pillai19} on the mixing time for the FA-$1$f
kinetically constrained model in the same regime (see Corollary
\ref{cor: FA}). Motivated by a
different application to the kinetically constrained models FA-$j$f
with $j>1$ (see \cite{Hartarsky20FA}), we then investigate a
version of the model in which the single vertex state space $\{0,1\}$
is replaced by an arbitrary finite set and we bound its mixing time
(see Theorem \ref{mainthm 3}).

\subsection{The CBSEP and $g$-CBSEP models}
\label{subsec:CBSEP}
Let $G=(V,E)$ be a finite connected graph with $n$ vertices.
The degree of $x \in V$ is denoted by $d_x$. Minimum, maximum, and
average degrees in $G$ are denoted by $d_{\rm min}, d_{\rm max}$ and
$d_{\rm avg}$, respectively. For any $\o\in\O=\{0,1\}^V$ and any vertex $x\in V$ we say that $x$ is \emph{filled/empty}, or that there is a \emph{particle/hole} at $x$, if $\o_x=1/0$. We define $\O_+=\O\setminus \{0\}$ to be the event that there exists at least one particle. Similarly, for any edge $e=\{x,y\}\in E$ we refer to $(\o_x,\o_y)\in \{0,1\}^{\{x,y\}}$ as the state of $e$ in $\o$ and write $E_e=\{\o\in\O\tc\o_x+\o_y\neq 0\}$ for the event that $e$ is not empty.

Given $p\in
(0,1)$, let $\pi=\bigotimes_{x\in V}\pi_x$ be the product Bernoulli measure, in which each vertex is filled with probability $p$, and let $\mu(\cdot):=\pi(\cdot\tc \O_+)$. Given an edge
$e=\{x,y\}$ we write
$\pi_e:=\pi_x\otimes\pi_y$ and $\l(p):=\pi(E_e)=p(2-p)$. In the
sequel we will always assume for simplicity that $p$ is bounded away
from $1$ (e.g.\ $p\le 1/2$).

The {CBSEP}, the main object of this work, is a continuous-time Markov chain on $\O_+$ for which the state of any edge $e\in E$ such that $E_e$ occurs  is resampled with rate one w.r.t.\ $\pi_e(\cdot\tc E_e).$ Thus, any edge containing exactly one particle with rate $(1-p)/(2-p)$ moves the particle to the opposite endpoint (the \emph{SEP move}) and
with rate $p/(2-p)$ creates an extra particle at its empty endpoint
(the \emph{branching move}). Moreover, any edge containing two
particles with rate $2(1-p)/(2-p)$ kills one of the two particles chosen uniformly (the \emph{coalescing move}).  The chain is readily seen to be reversible w.r.t.\ $\mu$ and
ergodic on $\O_+$, because it can reach the configuration with a
particle at each vertex.
If $c(\o,\o')$ denotes the jump rate from $\o$ to $\o'$, the Dirichlet
form $\cD(f)$ of the chain has the expression
\begin{equation}
\label{eq:def:cD}
\cD(f)=\frac 12 \sum_{\o,\o'}\mu(\o)c(\o,\o')\left(f(\o')-f(\o)
  \right)^2=\sum_{e\in E}\mu(\1_{E_e}\var_{e}(f\tc E_e)).\end{equation}
 It is also not hard to check that CBSEP is equivalent to the
 coalescing and branching random walks described in the introduction
up to a global time-rescaling.

We will also consider a \emph{generalised} version of  {CBSEP}, in the
sequel {$g$-{CBSEP}}, defined as follows. We are given a graph $G$ as
above together with a probability space $(S,\rho),$ where $S$ is a finite set and $\rho$ a probability measure on $S$. We still
write $\rho=\bigotimes_{x\in V}\rho_x$ for the product probability on $\O^{(g)}:=S^V$. In the \emph{state
  space} $S$, we are given a bipartition $S_1\sqcup S_0=S$, and
we write $p:=\rho(S_1)\in(0,1).$
We
define the  \emph{projection} $\f:\O^{(g)}\to \O=\{0,1\}^V$ by $\f(\o)=(\1_{\{\o_x\in
  S_1\}})_{x\in V}$ and we let $\O_+^{(g)}=\{\o\in \O:\ \sum_x
(\f(\o))_x\ge 1\}=\f^{-1}(\O_+)$. For any edge $e=\{x,y\}\in E$ we also let $E^{(g)}_e$
be the event that there exists a particle at $x$ or at $y$ for $\f(\o)$.
In {$g$-{CBSEP}} every edge $e=\{x,y\}$ such that $E^{(g)}_e$ is satisfied is resampled
with rate 1 w.r.t.\ $\rho_x\otimes\rho_y(\cdot\tc E_e^{(g)})$. A key property is
that its projection chain onto the variables $\f(\o)$ coincides with {CBSEP} on $G$ with parameter $p$.
As with CBSEP, the {$g$-{CBSEP}}  is 
reversible w.r.t.\ 
$\rho_+=\rho(\cdot\tc \O^{(g)}_+)$
and
ergodic on $\O^{(g)}_+$. For the main motivation behind $g$-CBSEP we refer the reader to Section \ref{sec:g-motivations}.

\subsection{The FA-$j$f KCM}
\label{subsec:FA}
We next define another class of models of interest---the $j$-neighbour or the Fredrickson-Andersen
$j$-facilitated kinetically constrained spin model (FA-$j$f KCM for short). In the
setting of Section \ref{subsec:CBSEP} for {CBSEP}, these chains evolve as
follows. With rate one and
independently w.r.t.\ the other vertices, the state of each vertex
$x\in V$ with at least $j$ neighbouring particles is resampled
w.r.t.\ $\pi_x$. In this paper we will focus on the simplest case $j=1$, the case $j=0$ being
trivial. As for {CBSEP} it is immediate to check that,
on $\O_+$, the chain is ergodic with $\mu$ as the unique reversible
measure, and that its Dirichlet form is
\[
\cD[FA](f)=\sum_x\mu\left({\1_{\left\{\sum_{\{x,y\}\in E}\o_y\ge 1\right\}}}p(1-p)(f(\o^x)-f(\o))^2\right),
  \]
  where $\o^x$ denotes the configuration $\o$ flipped at $x$.

The FA-$1$f KCM has been extensively studied (see e.g.\ \cites{Fredrickson84,Fredrickson85,Blondel13,
Cancrini08,
Cancrini09}). Of particular relevance for us are the beautiful works
of Pillai and Smith \cites{Pillai17,Pillai19} that we present
next. For any positive integers $d$ and $L$, set $n=L^d$, and let $\bbZ_L=\{0,1,\dots,
L-1\}$ be the set of remainders modulo $L$. The $d$-dimensional
discrete torus with $n$ vertices, $\bbT^d_n$ in the sequel, is the set $\bbZ_L^d$
endowed with the graph structure inherited from $\bbZ^d$. For the \emph{discrete time}
version of FA-$1$f on $\bbT_n^d$ with $p=c/n$ \cites{Pillai17,Pillai19} provide a rather
precise bound for the (total variation) mixing time $\tmix[FA]$. Translated into
the continuous time setting described above, their results read
\begin{equation}
\label{eq:PS}
\begin{aligned}
  C^{-1}n^2&\le \tmix[FA]\le C n^2\log^{14}(n)  &d&{}=2\\
  C^{-1}n^2&\le \tmix[FA]\le C n^2\log(n)  &d&{}\ge 3,
\end{aligned}
\end{equation}
where $C>0$ may depend on $d$ but not on $n$.
\begin{rem}
In \cite{Pillai17}*{Section 2} it was argued that $\tmix[FA]$ should
be lower bounded by the time necessary to get two well-separated particles starting from
one. By reversibility, and since to move an isolated particle by one step, we should first create a particle at a neighbouring site at rate $p$,
this time should correspond 
$p^{-1}\tmeet$, where $\tmeet$
is the meeting time of two independent continuous time
random walks on $\bbT_n^2$ with independent uniformly distributed starting points.
In particular, since in two dimensions $\tmeet=\Theta(n\log(n))$, in
\cite{Pillai19}*{Remark 1.1} it was conjectured that
$\tmix[FA]=\O(p^{-1}n\log(n))=\O(n^2\log(n))$ in the regime $p=\Theta(1/n)$ and this was recently confirmed by Shapira \cite{Shapira20}*{Theorem 1.2}. As it will be apparent in the proof of Theorem \ref{mainthm 1}\ref{lower:main}, this heuristics together with the attractiveness of {CBSEP} will allow us to prove a lower bound on the logarithmic Sobolev constant and relaxation time of {CBSEP} on a general graph.
\end{rem}

\subsubsection{Relationship between {CBSEP} and FA-$1$f}
Notice that the branching and coalescing moves of {CBSEP} are exactly the moves allowed in FA-$1$f. Moreover, the SEP move for the edge $\{x,y\}$ from $(1,0)$ to $(0,1)$ can be reconstructed using two consecutive FA-$1$f moves, the first one filling the hole at $y$ and the second one emptying $x$. If we
also take into account the rate for each move, we easily get the
following comparison between the respective Dirichlet forms (see e.g.\ \cite{Levin09}*{Ch. 13.4}): there exists an absolute constant $c>0$ such that for all $f:\Omega_+\to \bbR$ it holds that
  \begin{equation}
    \label{eq:13}
    c^{-1}\cD[FA](f)\le
    \cD(f)\le c d_{\rm max}p^{-1}  \cD[FA](f).
  \end{equation}
In our application of the main results for CBSEP  to the estimate of the mixing time of the FA-$1$f chain (see Corollary \ref{cor: FA}) for $p\ll 1$ only the upper bound, which we believe to be sharper, will count.

Although the two models are clearly closely related, we would like to
emphasise that {CBSEP} has many advantages over FA-$1$f, making its study
simpler. Most notably, {CBSEP} is \emph{attractive} in
the sense that there exists a grand-coupling (see e.g
\cite{Levin09}) which preserves the partial order on $\O$
given by $\o\prec \o'$ iff $\o_x\le \o'_x$ for all $x\in V$ (as it can be readily verified via the construction of Section \ref{sec:graphical}). Furthermore, it is also natural to embed in {CBSEP} a continuous time random walk $(W_t)_{t\ge 0}$ on $G$ such that {CBSEP} has a particle at $W_t$ for all $t\ge 0$. The latter is a particularly fruitful feature, which we will use in Section \ref{sec:main3}, and which is challenging to reproduce for FA-$1$f \cite{Blondel13}.

\subsubsection{Motivations for $g$-CBSEP: the FA-2f KCM}
\label{sec:g-motivations} 
The generalisation of CBSEP covering the case of general (finite) single vertex state space $S$ was introduced in \cite{Hartarsky20FA}*{Sections 1.5, 5.1} in order to model the effective random evolution of the \emph{mobile droplets} of the FA-$2$f KCM (see Section \ref{subsec:FA} for $j=2$). Modelling  the dynamics of mobile droplets as a suitable $g$-CBSEP combined with precise bounds on the relaxation time and probability of occurrence of the latter, proved to be a key tool to determine with very high precision the \emph{infection time} $\tau_0$ of the FA-2f KCM, i.e.\ the first time the state of the origin is $0$, as the density $p$ of the infection vanishes. In particular our Theorems \ref{mainthm 1} and \ref{mainthm 3} play a key role 
in \cite{Hartarsky20FA} for proving that, for the FA-2f model in
$\bbZ^2$ at density $p$, w.h.p.
\[\tau_0=\exp\left(\frac{\pi^2+o(1)}{9p}\right), \quad \text{as $p\to 0$}.\]

\section{Preliminaries}
In order to state our results we need first to recall some classical
material on \emph{mixing times} for finite Markov chains (see e.g.\
\cites{Saloff-Coste97,Hermon18}) and on the \emph{resistance distance}
on finite graphs (see \cites{Chandra89, Tetali91}, \cite{Levin09}*{Ch. 9} and \cite{Lyons16}*{Ch.2}).
\subsection{Mixing times and logarithmic Sobolev constant}
\label{sec:mixing}
Given a finite state space $\O$ and a uniformly positive probability measure $\mu$  on $\O$, let $(\o(t))_{t\ge 0}$
be a continuous time ergodic Markov chain on $\O$ reversible w.r.t.\ $\mu$, and write $P_\o^t(\o')=\bbP(\o(t)=\o'\tc
\o(0)=\o)$. Let also $h_\o^t(\cdot)=
P_\o^t(\cdot)/\mu(\cdot)$ be the relative density of the law
$P_\o^t(\cdot) $ w.r.t.\ $\mu$. The
total variation mixing time of the chain, $\tmix$, is defined as
\[
  \tmix=\inf\left\{t>0 :\max_{\o\in
    \O_+}\|P_\o^t(\cdot)-\mu(\cdot)\|_{\rm TV}\le 1/(2e)\right\},
\]
where $\|\cdot\|_{\rm TV}$ denotes the \emph{total variation distance} defined as
\[
  \|P_\o^t(\cdot)-\mu(\cdot)\|_{\rm TV} = \frac 12 \|h_\o^t(\cdot)-1\|_1,
\]
where $\|g\|^\a_\a= \mu(|g|^\a), \a\ge 1$. The $\ell^2$-mixing time
$T_2$ or, more generally, the $\ell^q$-mixing times $T_q$, $q\ge 1,$ are
defined by
\[
  T_q=\inf\left\{t>0: \max_{\o\in
    \O_+}\|h_\o^t(\cdot)-1\|_{q}\le 1/e\right\}.
\]
Clearly $\tmix \le T_q$ for $q>1$ and it is known that for all $1 <q \le \infty$ the $\ell^q$-convergence profile is determined entirely by that for $q=2$ (see e.g.\ \cite{Saloff-Coste97}*{Lemma 2.4.6}). Moreover,
(see e.g.\ \cite{Saloff-Coste97}*{Corollary 2.2.7}, 
\begin{equation}
  \label{eq:4}
\frac 12 \csob\le  T_2\le \csob (1+ \frac 14 \log\log(1/\mu_*)),
\end{equation}
where $\mu_*=\min_{\o\in \O_+}\mu(\o)$ and $\a$ is the \emph{logarithmic Sobolev
constant} defined as the inverse of the best constant $C$ in the
logarithmic Sobolev inequality valid for any $f:\O_+\to \bbR$
\begin{equation}
\label{eq:def:csob}
 \ent(f^2):= \mu(f^2\log(f^2/\mu(f^2)))\le  C\cD(f).
\end{equation}
Finally, the \emph{relaxation time} $\trel$ is then defined as the best constant $C$
in the Poincar\'e inequality
\begin{equation}
\label{eq:def:trel}
  \var(f)\le C \cD(f).
\end{equation}
It is not difficult to prove that $\trel\le \tmix$ and that (see e.g.\ \cite{Diaconis96}*{Corollary 2.11})
\begin{equation}
  \label{eq:2}
2\trel\le \csob \le (2+
\log(1/\mu_*))\times \trel.
\end{equation}
{\bf Notation warning.} In the sequel, unless otherwise indicated, all
the quantities introduced above will not carry any
additional label when referring to {CBSEP}. On the contrary, the same quantities referring to other chains, e.g.\ the FA-$1$f KCM
or {$g$-{CBSEP}}, will always carry an
appropriate superscript. 

\subsection{Resistance distance}
\label{sec:resistance distance}
Given a finite connected simple graph $G=(V,E)$, let $\vec E$ denote
the set of \emph{ordered pairs} of vertices forming an edge of $E$. For $\vec e=(u,v)\in\vec E$ we set $-\vec
e=(v,u)$. Given an anti-symmetric  function $\theta$ on $\vec E$ (that is $\theta(\vec e)=-\theta(-\vec e)$) and two
vertices $x,y$ we say that $\theta$ is a \emph{unit flow from $x$ to
  $y$} iff $\sum_{v: (u,v)\in \vec E}\theta((u,v))=0$ for all $u\notin \{x,y\}$
and $\sum_{v: (x,v)\in \vec E}\theta((x,v))=1$. The \emph{energy
of the flow} $\theta$ is the quantity
$\cE(\theta)=\frac 12 \sum_{\vec e\in\vec E}\theta(\vec e)^2$
and we set
\begin{equation}
  \label{eq:8}
\cR_{x,y}= \inf\{\cE(\theta):\theta \text{ is a
  unit flow from $x$ to $y$}\}.
\end{equation}
The Thomson principle \cite{Thomson67} states that
the infimum in \eqref{eq:8} is actually attained at a unique unit flow.

The quantity $\cR_{x,y}$ can be interpreted as the effective resistance
in the electrical network obtained by replacing the vertices of $G$
with nodes and the edges with unit resistances. In graph theory it is
sometimes referred to as the \emph{resistance distance}. It is also
connected to the behaviour of the simple random walk on $G$ via the
formula $2 |E|\cR_{x,y}=
C_{x,y},$ where $C_{x,y}$ is the expected
commute time between $x$ and $y$. Furthermore, if we let  $\trel[rw]$ be 
the relaxation time of the random walk,
the bound $\max_{x,y}\cR_{x,y}\le c \sqrt{\trel[rw] }/d_{\rm
  min}$ holds \cite{Oliveira19}*{Corollary 1.1} where $c>0$ is a universal constant (see also
\cite{Aldous02a}*{Corollary 6.21} for regular graphs). Finally, by taking the
shortest path between $x,y$ and the flow $\theta$ which assigns unit
flow to each edge of the path,  $\cR_{x,y}\le d(x,y),$ where $d(\cdot,\cdot)$ is the
graph distance, with equality iff $x,y$ are linked by a
unique path. In the sequel and for notation convenience we shall write $\bar
\cR_y$ for the spatial average $\frac 1n \sum_{x}\cR_{x,y}$. 
\begin{rem}
  \label{rem: resistance}
For later use we present bounds on $\max_y\bar \cR_y$ for certain
special graphs. If $G$ is the $d$-hypercube ($n:=|V|=2^d$) it follows
from \cite{Pippenger10} that $ \bar
\cR_{y}=\Theta(1/\log n)$ for all $y\in V$. If instead $G$ is the
regular $b$-ary tree with $b\ge 2$ then $\max_y\bar
\cR_y=\Theta(\log n)$. If $G$ is a uniform random $d$-regular graph with
$n\to \infty$, and $d$
independent of $n$, then w.h.p.\ $\trel[rw] =O(1)$ \cites{Broder87,Friedman89}, and
therefore w.h.p.\ $\bar \cR_y=\Theta(1)$ for all $y\in
V$. 
Finally, if $G$ is the discrete $d$-dimensional torus $\bbT_n^d\subset \bbZ^d$ with $n$ vertices, then, as $n\to \infty$ and $d$ is fixed, it follows from \cite{Lyons16}*{Proposition 2.15} that
  \[
    \max_{y}\bar \cR_{y}=\Theta(1)\times
    \begin{cases}
      n&\text{if $d=1$},\\
      \log(n) & \text{if $d=2$},\\
      1 & \text{if $d\ge 3$}.
     \end{cases}
   \]  
\end{rem}

\section{Main results}
Our first theorem establishes upper and lower bounds for the inverse of the logarithmic Sobolev constant, $\csob$, and relaxation time, $\trel$, of CBSEP in the general setting described
in the introduction. 

Let $\tmeet$ denote the expected meeting time for two
\emph{continuous time} random walks jumping along each edge at rate $1$  and
started from two uniformly chosen vertices of $G$. We refer the reader
to \cites{Aldous02a,Kanade19} for the close connections between $\tmeet$ and $\cR_{x,y}$. Let also 
$\tmix[rw]$ denote the mixing time of the \emph{discrete time} lazy simple random walk on $G$ (i.e.\ staying at its position with probability $1/2$).

\begin{mainthm}
  \label{mainthm 1} Let $p_n\in(0,1)$ and consider {CBSEP} with parameter $p_n$ on a sequence of graphs $G=G_n=(V_n,E_n)$ with $|V_n|=n,$ maximum degree $d_{\rm max}=d_{\rm max}(n),$ minimum degree $d_{\rm min}=d_{\rm min}(n),$ and average degree $ d_{\rm avg}=d_{\rm avg}(n)$.  \begin{enumerate}[label=(\alph*)]
\item\label{upper:easy} If $p_n=\O(1)$, then
\begin{align}\label{eq:upper:easy}\csob&{}\le O(n)\\
\label{eq:upper:easy:trel}\trel&{}\le O(1).
\end{align}
\item\label{upper:main} If $p_n\to0$, then for some absolute constant $c>0$
\begin{align}
\csob&{}\le c \max\left(\frac{d_{\rm avg}
       d_{\rm max}^2}{d_{\rm min}^2}  \tmix[rw]\log(n),
      \left(\max_y  \bar \cR_{y}\right) n|\log (p_n)|\right)
    \label{eq:upper:main}\\
\trel &{}\le c n\max_y  \bar \cR_{y}.
    \label{eq:upper:main:trel}
  \end{align}
\item\label{lower:easy} There exists an absolute constant $c>0$ such that for all $p_n\in(0,1)$
\begin{align}
\label{eq:lower:easy}
\csob&{}\ge \frac{cn}{d_{\rm avg}}\\
\trel&{}\ge \frac{1-\m(\sum_x\o_x=1)}{pd_{\rm avg}}.
\label{eq:lower:easy:trel}
\end{align}
\item\label{lower:main} If $p_n=O(1/n)$, then we have the stronger bound
\begin{align}
\label{eq:lower:main}\csob&{}\ge \tmeet\O(1+|\log(np_n)|)\\
\trel&{}\ge \tmeet\O(1).
\label{eq:lower:main:trel}
\end{align}
\end{enumerate}
\end{mainthm}

For the reader's convenience, and in view of our application to the FA-$1$f KCM, we detail the above bounds for the graphs discussed in Remark
\ref{rem: resistance} when $p_n=\Theta(1/n)$.

\begin{cor}
  \label{cor: 1}
In the setting of Theorem \ref{mainthm 1} assume that
$p_n=\Theta(1/n)$. Then:
\begin{enumerate}[label=(\arabic*)]
\item\label{case1} hypercube:
\[\Theta\left(\frac{n}{\log n}\right)=2\trel\le \csob\le O(n),\]
\item\label{case2} regular $b$-ary tree, $b\ge 2$ independent of $n$:
\[\Theta(n\log(n))=2\trel\le \csob\le O(n \log^2(n)),\]
\item\label{case3} uniform random $d$-regular graph, $d$ independent of $n$: w.h.p.
\[\Theta(n)=2\trel\le \csob\le O(n \log(n)),\]
\item\label{case4} discrete torus $\bbT_n^d$ with $d$ independent of $n$: 
\begin{equation*}
  \csob\le   O(1)\times\begin{cases}
  n^2\log(n)
& d=1,\\
  n\log^2(n)& d=2,\\
  n\log(n)& d\ge 3,
  \end{cases}
  \end{equation*}
and
\[
  \a^{-1}\ge 2\trel=\Theta(1)\times\begin{cases}
    n^2 & d=1,\\
    n\log(n) & d=2,\\
    n & d\ge 3.
  \end{cases}
  \]
\end{enumerate}
\end{cor}
The corollary follows immediately from Theorem \ref{mainthm 1}\ref{upper:main} and \ref{lower:main} together with Remark \ref{rem: resistance}, the well-known results on \tmix[rw] for each graph and the fact that  (see
\cites{Aldous91,Chandra89,Kanade19}) for the graphs in Remark \ref{rem: resistance} it holds that
\[\tmeet=\Theta(n)\max_y \bar\cR_y.\]
Indeed, the upper bounds on \tmix[rw] are only needed to see that for these graphs the maximum in the r.h.s.\ of \eqref{eq:upper:main} is achieved by the second term.
Using Corollary \ref{cor: 1} together with \eqref{eq:4} and
\eqref{eq:13}, we immediately get the following consequences for the
FA-$1$f KCM to be compared with the r.h.s. of \eqref{eq:PS}.
\begin{cor}
\label{cor: FA}Consider the FA-$1$f KCM on $G=\bbT_n^d$ with parameter
$p_n=\Theta(1/n)$ and let $\tmix[FA]$ and $\ttwo[FA]$ be its mixing time and
$\ell^2$-mixing time respectively. Then
\begin{align}
\label{eq:PS:improvement}
  \tmix[FA]&\le \ttwo[FA]\le \csob[FA]\log (n)\le c n
             \log(n)\a^{-1}\le O(1)\times   \begin{cases}
    n^3\log^2(n) & d=1\\
    n^2\log^3(n) & d=2\\
  n^2\log^2(n)  & d\ge 3.
  \end{cases}
  \end{align}
\end{cor}
\begin{rem}Our results in $d\ge 2$, besides being more directly proved than in \cites{Pillai17,Pillai19}, hold in the stronger
logarithmic Sobolev sense, and extend to other graphs, e.g.\ all the graphs discussed in
Corollary \ref{cor: 1}. Furthermore, contrary to the approach followed 
in \cites{Pillai17,Pillai19},
our methods can be easily adapted to cover other regimes of $p_n$. For $d=1$ the above upper bound on $\tmix[FA]$ can be proved to also be sharp up to logarithmic corrections, using the technique discussed in \cite{Cancrini08}*{Section 6.2}.
\end{rem}

Our second theorem concerns the total variation mixing time of the generalised model, {$g$-{CBSEP}}.

Let $\t_{\rm cov}$ denote the cover time of the simple random walk on $G$
(see e.g.\ \cite{Levin09}*{Chap. 11} and also \cite{Chandra89} for a close connection between the average cover time and the resistance distance), and let
\[\tcov=\inf\left\{t>0:\max_{x\in V} \bbP_x(\t_{\rm cov}>t)\le 1/e\right\}.
\]

\begin{mainthm}
\label{mainthm 3} Consider {$g$-{CBSEP}} on a finite connected graph $G$ of minimum degree $d_{\rm min}$ with parameter
$p=\rho(S_1)$ and let $\tmix$ be the mixing time of {CBSEP} on $G$ with
parameter $p$. Then there
exists a universal constant $c>0$ such that
\[
\tmix \le \tmix[\text{$g$-CBSEP}]\le c (\tmix + \tcov/d_{\rm \min}).
\]
\end{mainthm}

The main reason to bound the total variation
mixing time of {$g$-{CBSEP}}, instead of the $\ell^q$-mixing times as for {CBSEP}, is that the scaling
of the logarithmic Sobolev constant for {$g$-{CBSEP}} is very
different from that of the {CBSEP}, as the following example shows. 
\begin{ex}
Let $G=\bbT_n^2$, $p_n=1/n$, $S=\{0,1,2\}$, and $\rho(1)=p, \rho(0)=\rho(2)=(1-p)/2.$ Then, 
\begin{equation}
  \label{eq:6}
\csob[\text{$g$-CBSEP}]=n^{3/2+o(1)}.  
\end{equation}
In the same setting Corollary \ref{cor: 1} gives $\a^{-1}=O(n\log^2 (n))$. To prove \eqref{eq:6} it is enough to take as test function in the logarithmic Sobolev inequality for {$g$-{CBSEP}} the indicator that a vertical strip of width $\lfloor \sqrt{n}/2\rfloor $ of the torus is in state $0$.
\end{ex}

\section{{CBSEP}---Proof of Theorem \ref{mainthm 1}}
\label{sec:main1}
For this section we work with {CBSEP} in the setting of Theorem \ref{mainthm 1} and abbreviate $p=p_n$. In the sequel $c$ shall denote an absolute constant whose value may
change from line to line.

\subsection{Upper bounds---Proof of Theorem \ref{mainthm 1}\ref{upper:easy} and \ref{mainthm 1}\ref{upper:main} }
Let us first prove the easy upper bound Theorem \ref{mainthm 1}\ref{upper:easy}, assuming that $p=\O(1)$. We know from \cite{Cancrini09}*{Theorem 6.4} that $\trel[FA]=O(1)$. Recalling \eqref{eq:13} and the definition of the relaxation time, we get that for {CBSEP} $\trel=O(1)$, yielding \eqref{eq:upper:easy:trel}. By \eqref{eq:2} this gives $\csob=O(n)$ and concludes the proof of Theorem \ref{mainthm 1}\ref{upper:easy}.

The rest of this section is dedicated to the proof of the main upper bound---Theorem \ref{mainthm 1}\ref{upper:main}. The starting point is the following decomposition of the entropy of any $f:\O_+\to \bbR$
      \begin{equation}
        \label{eq:5}
      \ent(f^2)= \mu\left(\ent(f^2\tc N)\right) + \ent\left(\mu(f^2\tc N)\right),
            \end{equation}
where $N(\o)=\sum_{x\in V}\o_x$ is the number of particles and $\ent(f^2\tc N)$ is the entropy of $f^2$ w.r.t.\ the conditional measure $\mu(\cdot\tc N)$ (see \eqref{eq:def:csob}).
The first term in the r.h.s.\ above is bounded from above using the logarithmic Sobolev constant of the SEP on $G$ with a fixed number of particles.
    \begin{prop}
    \label{prop:step1}
      There exists an absolute constant $c>0$ such that
      \[
        \mu\left(\ent(f^2\tc N)\right) \le c \log(n)
        \frac{d_{\rm avg}d_{\rm max}^2}{ d_{\rm min}^2} \tmix[rw]\cD(f),
      \]
   \end{prop}
    \begin{proof}
Let
      \[
        \cD_G^{\rm SEP}(f) =
       \frac 12 \sum_{e\in E}\mu\left(\left(f(\o^e)-f(\o)\right)^2\right),
      \]
      where $\o^e$ is the
  configuration obtained from $\o$ by swapping the states at the
  endpoints of the edge $e,$ denote the Dirichlet form of the symmetric
  simple exclusion process on $G.$ Similarly let
  \[
    \cD_{K_n}^{\rm
      BL}(f)=\frac{1}{2n}\sum_{e\in E(K_n)}\mu\left(\left(f(\o^{e})-f(\o)\right)^2\right)
  \]
be the Dirichlet form of the Bernoulli-Laplace process on the complete graph
$K_n$. The main result of \cite{Alon20}*{Theorem 1} implies
that\footnote{Actually the comparison result proved in \cite{Alon20} is much stronger, since it concerns weighted exchange processes on $G$ and on $K_n$.} 
\[
\cD[BL][K_n](f)\le  c
\frac{2|E|}{n}\frac{d_{\rm max}^2}{ d_{\rm min}^2 } \tmix[rw]\cD[SEP][G](f). 
  \]
Using $2|E|=\sum_x d_x$ we get, in particular, that
  \[
    \cD[BL][K_n](f)\le  c \frac{d_{\rm avg}d_{\rm max}^2}{ d_{\rm min}^2}  \tmix[rw]\cD[SEP][G](f).
  \]
On other hand, the logarithmic Sobolev constant of the Bernoulli-Laplace process on
$K_n$ with $k\in \{1,\dots,n-1\}$ particles is bounded by $c\log n$ uniformly in $k$ \cite{Lee98}*{Theorem 5}. Hence,
\[\mu\left(\ent(f^2\tc N)\right)\le c\log(n) \cD[BL][K_n](f) \le c
  \log(n) \frac{d_{\rm avg}d_{\rm max}^2}{ d_{\rm min}^2 } \tmix[rw] \cD[SEP][G](f).\]
The proposition then follows using $p\le 1/2$ and
  \[
    \cD[SEP][G](f)\le \frac{(2-p)}{(1-p)} \cD(f).\qedhere
    \]
    \end{proof}
 We now examine the second term $\ent(\mu(f^2\tc N))$ in the r.h.s.\ of \eqref{eq:5}. Let
    \[
      g(k):=
      \mu\left(f^2\tc N=k\right)^{1/2}
    \]
for $k\ge 1$, so that $\ent(\mu(f^2\tc N))=\ent_\g(g^2),$ where $\g$ is the probability law of $N$ on $\{1,\dots,n\}$. Clearly, $\g$ is ${\rm Bin}(n,p)$ conditioned to be positive, so that for any $2\le k\le n$ we have
\begin{equation}
\label{eq:bin:ratio}
\g(k)(1-p)k=\g(k-1)p(n-k+1).
\end{equation}
 \begin{prop}
 \label{prop:ent}
There exists an absolute constant $c>0$ such that 
\[     \ent_\g(g^2)\le c \log(1/p)\times p\sum_{y\in V}
    \mu\left(\left[f(\o^y)-f(\o)\right]^2 (1-\o_y)\right).\]
 where we recall that $\o^y$ denotes the configuration $\o$ flipped at $y$.
 \end{prop}
 
 \begin{proof}
The proof starts with a logarithmic Sobolev inequality for $\g$
w.r.t.\ a suitably chosen reversible death and birth process on
$\{1,\dots,n\}$.
\begin{lem}
  \label{lem:birth-death}
There exists an absolute constant $c>0$ such that for any non-negative function $g:\{1,\dots,n\}\to \bbR$
\[
     \ent_\g(g^2)\le c \log(1/p)\times \sum_{k=2}^n \g(k) k [g(k)-g(k-1)]^2.\]
        \end{lem}
Leaving the tedious proof to the \hyperref[appn]{Appendix}, we move on to bounding the r.h.s.\ above for the special choice $g=\mu(f^2\tc N)^{1/2}$.
\begin{claim}
\label{claim:g:diff:squared}
For any $2\le k\le n$ we have
        \[
\left(g(k)-g(k-1)\right)^2\le \frac{A^2_k}{g^2(k-1)+ g^2 (k)},
  \]
  where
\begin{equation*}
A_k=\frac{1}{n-k+1}\sum_{y\in V} \mu\left((1-\o_y)\left[f^2(\o)-f^2(\o^y)\right]\tc N=k-1\right).
\end{equation*}
\end{claim}
        \begin{proof}
We first observe that
\begin{equation}
\label{eq:g:diff:squared}
[g(k)-g(k-1)]^2=\frac{[g^2(k)-g^2(k-1)]^2}{[g(k)+g(k-1)]^2}\le \frac{[g^2(k)-g^2(k-1)]^2}{g^2(k)+ g^2(k-1)}.
\end{equation}
       Next we write
       \begin{align*}
         g^2(k-1)&{}= \sum_{\o:\ N(\o)=k-1}\frac{\mu(\o)}{\g(k-1)}f^2(\o)\\
         &{}= \frac{1}{n-k+1}\frac{1}{\g(k-1)}\sum_{y\in V}\sum_{\o:\ N(\o)=k-1}\mu(\o)(1-\o_y)f^2(\o).
       \end{align*}
With the change of variable $\h=\o^y$ we get that the r.h.s.\ above
is equal to
\begin{multline*}
\frac{1}{n-k+1}\frac{1}{\g(k-1)}\sum_{y\in V}\sum_{\o:\
     N(\o)=k-1}\mu(\o)(1-\o_y)\left[f^2(\o)-f^2(\o^y)\right]\\
   + \frac{\g(k)(1-p)}{p(n-k+1)\g(k-1)}\sum_{y\in V}\sum_{\eta:\
     N(\eta)=k}\frac{\mu(\eta)}{\g(k)}\eta_yf^2(\eta),
     \end{multline*}
the second line being equal to $g^2(k)$ by \eqref{eq:bin:ratio}. In conclusion $g^2(k-1)= g^2(k) + A_k$ 
and the claim follows from \eqref{eq:g:diff:squared}.
\end{proof}
\begin{claim}
\label{claim:Ak}
For any $2\le k\le n$ we have
\[A_k^2 \le \frac{2\left(g^2(k-1)+g^2(k)\right)}{n-k+1}\sum_{y\in V} \mu\left(\left[f(\o)-f(\o^y)\right]^2 (1-\o_y)\tc N=k-1\right).\]
    \end{claim}
    \begin{proof}
Using $f^2(\o)-f^2(\o^y)=(f(\o)-f(\o^y))(f(\o)+f(\o^y)$ and the Cauchy-Schwarz inequality
 w.r.t.\  $\mu\left(\cdot\tc  N=k-1, \o_y=0\right)$, we get 
\begin{multline*}
A_k
\le 
\text{Av}\Big(
\mu([f(\o)-f(\o^y)]^2\tc N=k-1, \o_y=0)^{1/2} \times \\
\mu([f(\o)+f(\o^y)]^2\tc N=k-1, \o_y=0)^{1/2}\Big),
\end{multline*}
where for any $h:V\to \bbR$
\[
  \text{Av}(h):=\frac{1}{n-k+1}\sum_{y\in V}\mu\left((1-\o_y)\tc N=k-1\right)
  h(y).
\]
Another application of the Cauchy-Schwarz inequality, this time w.r.t.\ $\text{Av}(\cdot),$ gives
\begin{multline*}
  A_k^2\le 
  \frac{1}{n-k+1}\sum_{y\in V} \mu\left(\left[f(\o)-f(\o^y)\right]^2(1-\o_y)\tc
  N=k-1 \right)\times  \\
 \frac{2}{n-k+1}\sum_{z\in V} \mu\left(\left[f^2(\o)+f^2(\o^z)\right](1-\o_z)\tc
 N= k-1\right).\end{multline*}
Inside the second factor in the above r.h.s.\ the term containing $f^2(\o)$ is equal to $2\mu(f^2\tc N=k-1)=2g^2(k-1)$. Similarly, the term
containing $f^2(\o^y)$, after the change of variable $\eta=\o^y$ and recalling \eqref{eq:bin:ratio},
equals
\[
 \frac{2}{n-k+1}\frac{\g(k)k (1-p)}{p\g(k-1)}\mu\left(f^2(\eta)\tc
 N=k\right)=2g^2(k).\qedhere
\]
 \end{proof}
Combining Claims \ref{claim:g:diff:squared} and \ref{claim:Ak}, we get that 
\begin{equation}
\begin{aligned}  
        \left(g(k)-g(k-1)\right)^2&{}\le \frac{A^2_k}{g^2(k-1)+ g^2 (k)}\\
    \label{eq:1}    &{}\le
        \frac{2}{n-k+1}\sum_{y\in V} \mu\left(\left[f(\o)-f(\o^y)\right]^2 (1-\o_y)\tc N=k-1\right).
\end{aligned}
\end{equation}

Using \eqref{eq:1}  
together with \eqref{eq:bin:ratio}, we get
\begin{multline*}
\sum_{k=2}^n \g(k) k [g(k)-g(k-1)]^2\\
\begin{aligned}
\le{}& \sum_{k=2}^n \frac{2k\g(k)}{n-k+1}\sum_{y\in V} \mu\left(\left[f(\o)-f(\o^y)\right]^2 (1-\o_y)\tc N=k-1\right)\\
    ={}&\frac{2}{1-p}p\sum_{y\in V}
    \mu\left(\left[f(\o)-f(\o^y)\right]^2 (1-\o_y)\right).
\end{aligned}
\end{multline*}
Using the above bound together with Lemma \ref{lem:birth-death} we
get the statement of Proposition \ref{prop:ent}.
\end{proof}
The final step in the proof of \eqref{eq:upper:main} is the
following comparison between the quantity  $p\sum_{y\in V}
    \mu\left(\left[f(\o)-f(\o^y)\right]^2 (1-\o_y)\right) $ and the
    Dirichlet form $\cD(f)$ using electrical networks. Recall the
    definition of the resistance distance and of $\max_y\bar \cR_y$
    given in Section \ref{sec:resistance distance}.
\begin{prop}
\label{prop:electricity}
\[p\sum_{x\in V}\mu((f(\o^x)-f(\o))^2(1-\o_x))\le 4n\max_{y\in V}\bar \cR_y\times\cD(f).\]
\end{prop}
\begin{proof}
We will identify $\o\in\{0,1\}^{V}$ with its set of particles
$\{x\in V: \o_x=1\}$ and we set $F_\o(u):=f(\o\cup\{u\}), u\in V.$ For
each $\vec e=(u,v)\in\vec E$ we also write $\nabla_{\vec e} F_\o:=F_\o(v)-F_\o(u)$. 
Given $x\in V$ and $\o\in\O_+$,
let $y_\o\in V$ be an arbitrarily chosen vertex such that $\o_{y_\o}=1$, and let $\theta_*$ be the
optimal (i.e.\ with the smallest energy)
unit flow from $x$ to $y_\o$. By applying \cite{Lyons16}*{Lemma 2.9}
to the function $F_\o$ and
using the
Cauchy-Schwarz inequality, we get that for any $\o\in\O_+$ and $x\in V$ such that $\o_x=0$
\begin{align*}
  (f(\o^x)-f(\o))^2 &=(F_\o(x)-F_\o(y_\o))^2\\
  &=\left(\frac 12 \sum_{\vec e\in\vec E} \theta_*(\vec e)\nabla_{\vec
    e}F_\o\right)^2\le \cE_{x,y_\o}\times  \frac 12 \sum_{\vec e\in\vec E}(\nabla_{\vec e}F_\o)^2.
  \end{align*}
  Hence,
  \[
    \sum_{x\in V} \left(f(\o^x)-f(\o)\right)^2 (1-\o_x)\le n\left(\max_{y\in V} \bar \cR_y\right)\times \frac 12 \sum_{\vec e}(\nabla_{\vec e}F_\o)^2.
  \]
  We next transform the generic term in the sum above into a Dirichlet
form term for {CBSEP}. For any $\vec e=(u,v)\in\vec E$ we have
\begin{multline*}
  p\mu(\o)(\nabla_{\vec e}F_\o)^2 \\
  =\mu(\o\cup\{u\})\times\begin{cases}0&\{u,v\}\subset\o\\
p(f(\o\cup \{v\})-f(\o))^2&u\in\o\not\ni v\\
(1-p)(f(\o\cup \{v\})-f(\o\cup \{u\}))^2& \{u,v\}\cap\o=\varnothing.
\end{cases}
 \end{multline*}
 
Comparing with the expression of $\cD(f)$, \eqref{eq:def:cD}, we get
immediately that 
\[\frac 12\sum_{\o\in\O_+}p\mu(\o) \sum_{\vec e\in\vec E}(\nabla_{\vec e}
  F_\o)^2\le 4\cD(f).\qedhere\]
\end{proof}

We are now ready to prove \eqref{eq:upper:main}. Using Proposition \ref{prop:step1} the first term in the r.h.s.\ of \eqref{eq:5} is bounded from above by 
\[
c \log(n)
        \frac{d_{\rm avg}d_{\rm max}^2}{ d_{\rm min}^2} \tmix[rw]\cD(f).
\]
In turn, Proposition \ref{prop:ent} combined with  Proposition \ref{prop:electricity} gives that the second term in the r.h.s.\ of \eqref{eq:5} is bounded from above by 
\[
c\log(1/p)\times 4n\max_{y\in V}\bar \cR_y\times\cD(f).
\]
In conclusion,
\begin{align*}
    \ent(f^2)\le c\max\left(\log(n)
        \frac{d_{\rm avg}d_{\rm max}^2}{ d_{\rm min}^2} \tmix[rw], 
        \log(1/p)\times 4n\max_{y\in V}\bar \cR_y\right)\times \cD(f),
\end{align*}
so that the best constant in the logarithmic Sobolev inequality \eqref{eq:def:csob}  satisfies \eqref{eq:upper:main}.

Turning to \eqref{eq:upper:main:trel}, Proposition \ref{prop:electricity} alone is enough to conclude. Indeed, using the two-block argument of \cite{Blondel13}*{Lemma 6.6} (see also Lemma 6.5 and Proposition 6.2 therein) and the well-known fact that the variance w.r.t.\ a product measure is at most the average of the sum of variances over single spins (see e.g.\ \cite{Ane00}*{Chapter 1}), we get
\[\var(f)\le cp\sum_{x\in V}\mu((f(\o^x)-f(\o))^2(1-\o_x)).\]
The desired bound \eqref{eq:upper:main:trel} then follows from \eqref{eq:13}, \eqref{eq:def:trel} and Proposition \ref{prop:electricity}.

\subsection{Lower bounds---Proof of Theorem \ref{mainthm 1}\ref{lower:easy} and  \ref{mainthm 1}\ref{lower:main}}
Inject $f=\1_{\{N=1\}}$, the indicator of having exactly one particle, in the logarithmic Sobolev inequality \eqref{eq:def:csob}. For $c>0$ small enough we have
\[\frac{\ent(f^2)}{\cD(f)}=\frac{\mu(N=1)|\log(\mu(N=1))|}{\frac{2|E|\mu(N=1)}{n}\cdot\frac{p}{2-p}}\ge\frac{|\log \mu(N=1)|}{pd_{\rm avg}}\ge c\frac{n}{d_{\rm avg}},\]
since $\mu(N=1)=np(1-p)^{n-1}/(1-(1-p)^n)$. To check the last inequality, one may distinguish the cases $np$ sufficiently large/of order 1/sufficiently small. This proves \eqref{eq:lower:easy}. Using the same function, so that $\var(f)=\m(N=1)(1-\m(N=1))$, we obtain \eqref{eq:lower:easy:trel} in the same way. This concludes the proof of Theorem \ref{mainthm 1}\ref{lower:easy}.

The rest of this subsection is dedicated to the proof of the main lower bound---Theorem \ref{mainthm 1}\ref{lower:main}, so we assume that $p_n=O(1/n)$. Let $\l_0>0$ be the smallest eigenvalue, restricted to the event
$\{N\ge 2\}$ that there are at least two particles, of $-\cL$, where $\cL$ is the generator of {CBSEP}. By \cite{Goel06}*{Lemma 4.2, Equation (1.4)} we have that 
\begin{align*}
\csob&{}\ge \l_0^{-1}|\log(\m(N\ge 2))|,\\
\trel&{}\ge \l_0^{-1}(1-\mu(N\ge 2)),
\end{align*}
the second inequality being easy to check from the definition. It is well known (see e.g.\ \cite{Hermon18}*{Section 3.4}) that
\[
\l_0^{-1}\ge\bbE_{\mu(\cdot\tc N\ge 2)}(\t),
\]
where $\t$ is the first time when $N=1$. Putting these together and recalling that $p_n=O(1/n)$, we obtain
\begin{align*}
\csob&{}\ge \bbE_{\m(\cdot\tc N\ge 2)}(\t)|\log(\m(N\ge 2))|\ge \bbE_{\m(\cdot\tc N\ge 2)}(\t)\O(1+|\log(np_n)|),\\
\trel&{}\ge \bbE_{\mu(\cdot\tc N\ge 2)}(\t)\m(N=1)\phantom{|\log()|}\ge \bbE_{\mu(\cdot\tc N\ge 2)}(\t)\O(1).
\end{align*}
In turn, again using that $p_n=O(1/n)$, we get
  \[
    \bbE_{\mu(\cdot\tc N\ge 2)}(\t)\ge \mu(N=2|N\ge 2)\bbE_{\mu(\cdot\tc N=2)}(\t)\ge \O(1)\bbE_{\mu(\cdot\tc N=2)}(\t).\]

It is not hard to see (e.g.\ via a graphical construction---see Section \ref{sec:graphical}) that  {CBSEP} stochastically dominates a process of coalescing random walks with birth rate $0$, which we will call CSEP. Therefore, $\bbE_\o(\t)\ge \bbE^{\mathrm{\scriptstyle{CSEP}}}_\o(\t)$ for any $\o\in\O_+$. Furthermore, {CSEP} started with two particles has the law of two independent continuous time random walks  which jump along each edge with rate $(1-p)/(2-p)$ and coalesce when they meet. Hence, we obtain \eqref{eq:lower:main} and \eqref{eq:lower:main:trel}, concluding the proof of Theorem \ref{mainthm 1}\ref{lower:main}.

\section{{$g$-{CBSEP}}---Proof of Theorem \ref{mainthm 3}}
\label{sec:main3}
\subsection{Graphical construction}
\label{sec:graphical}
We start by introducing a graphical construction of ${g}$-{CBSEP} for all initial conditions. The graphical construction of CBSEP can then be immediately deduced by considering the 
special case $S_1:=\{1\}$ and $S_0:=\{0\}$.

To each edge $e\in E$ we associate a Poisson process of rate $p/(2-p)$ of arrival times $(t_n^e)_{n=1}^\infty$. Similarly, to each oriented edge $\vec e\in \vec E$ we associate a Poisson process of rate $(1-p)/(2-p)$ of arrival times $(t_n^{\vec e})_{n=1}^\infty$. All the above processes are independent as $e,\vec e$ vary in $E,\vec E$ respectively. 
Furthermore, for $e\in E$, $\vec e\in\vec E$ and $n\ge 1$, we define $X_n^e$ and $X_n^{\vec e}$ to be mutually independent random variables taking values in $S^2$. We assume that for all $n$ and $(u,v)\in\vec E$, the law of $X_n^{(u,v)}$ is $\rho_u(\cdot\tc S_1)\otimes\rho_v(\cdot\tc S_0)$. Similarly, for $\{u,v\}\in E$, the law of $X_n^{\{u,v\}}$ is $\r_u(\cdot\tc S_1)\otimes\r_v(\cdot\tc S_1)$. Given an initial configuration $\o(0)\in\O^{(g)}$ and a realization of the above variables, we define the realization of {$g$-{CBSEP}} $\o(t)$ as follows. 

Fix $t\ge 0$, let $t^*$ be the first arrival time after $t$, and let $\{x,y\}$ be the endpoints of the edge where it occurs. We set $\o_z(t^*)=\o_z(t)$ for all $z\in V\setminus\{x,y\}$. If $E_{\{x,y\}}^{(g)}$ does not occur, that is $\o_x(t)\in S_0$ and $\o_y(t)\in S_0$, we set $\o(t^*)=\o(t)$. Otherwise, we set \[(\o_x(t^*),\o_y(t^*))=\begin{cases}X^{\{x,y\}}_n&\text{if }t^*=t_n^{\{x,y\}},\\
X^{(x,y)}_n&\text{if }t^*=t_n^{(x,y)}.
\end{cases}\]

\begin{obs}\label{obs:coupling}
Let $\o(t)$ and $\o'(t)$ be two {$g$-{CBSEP}} processes constructed using the same Poisson processes $(t_n^e)_{n=1}^\infty$, $(t^{\vec e}_n)_{n=1}^\infty$ and variables $X_n^e$, $X_n^{\vec e}$ above, but with different initial conditions $\o,\o'\in\O^{(g)}$ satisfying $\f(\o)=\f(\o')=\h\in\O$. Fix $t\ge 0$ and let $\cF_t$ be the sigma-algebra generated by the arrival times smaller than or equal to $t$ (but not the $X_n^e$ and $X_n^{\vec e}$ variables). Then $\f(\o(t))=\f(\o'(t))=:\h(t)$ is $\cF_t$-measurable and only depends on $\o$ through its projection $\h$. 

We say that a vertex $v\in V$ is \emph{updated} if $v\in e\in E$ so that there exists $0\le t^*\le t$ and $n$ such that $t^*\in\{t^{e}_n,t^{\vec e}_n,t^{-\vec e}_n\}$ and the event $E_e^{(g)}$ occurs for $\o(t^*)$, i.e.\ a successful update occurs at $v$. Denoting the set of updated vertices by $\X_t$, we have
\begin{itemize}
\item $\X_t$ is $\cF_t$-measurable and only depends on $\o$ through its projection $\h$,
\item if $x\in\X_t$, then $\o_x(t)=\o'_x(t)$ and, conditionally on $\cF_t$, the law of $\o_x(t)$ is $\r(\cdot\tc S_{\h_x(t)})$,
\item if $x\in V\setminus\X_t$, then $\o_x(t)=\o_x(0)$ and, in particular, $\h_x(t)=\h_x$.
\end{itemize}
In particular, for all $x\in V$ such that there exists $t_x\le t$ satisfying $(\f(\o(t_x)))_x\neq (\f(\o(0)))_x$, we have $\o_x(t)=\o_x'(t)$ (since $x\in\X_t$).
\end{obs}

\subsection{Proof of Theorem \ref{mainthm 3}}
We are now ready to prove Theorem \ref{mainthm 3}. The lower bound is an immediate consequence of the fact that the projection chain on the variables $\f(\o)$ coincides with CBSEP. 

For the upper bound, let $\m_t^\h$ be the law of the {CBSEP} $\h_t$ at time $t$ with parameter $p=\r(S_1)$ and starting point $\h\in \O_+$. Further denote $\nu^\eta=\r(\cdot\tc \f(\o) =\eta)$, the measure $\r$ conditioned on whether or not a particle is present at each site. Since $\r$ is itself product, we have 
\begin{equation}
\label{eq:nueta}\n^\h=\bigotimes_{x\in V}\r_x(\cdot\tc S_{\h_x})=\bigotimes_{x:\h_x=1}\r_x(\cdot\tc S_1)\otimes\bigotimes_{x:\h_x=0}\r_x(\cdot\tc S_0).
\end{equation}
\begin{claim}
\label{claim:law}
The law $\rho_t^{\nu^\eta}$ of 
{$g$-{CBSEP}} with initial law $\nu^\eta$ at time $t$ takes the form  
\begin{equation}
\label{eq:rtnh}
\rho_t^{\nu^\eta}(\cdot)= \mu_t^\eta\big(\nu^{\eta_t}(\cdot)\big),
\end{equation}
i.e.\ it is the average of $\n^{\h'}$ over $\h'$ distributed as the {CBSEP} configuration $\h_t$ started from $\h$ at time $t$.
\end{claim}  
\begin{proof}
Fix $\o$ satisfying $\f(\o)=\h$. Denote by $\bbP^\o$ the probability w.r.t.\ the graphical construction of {$g$-{CBSEP}} of Section \ref{sec:graphical} with initial condition $\o$, by $\cF_t$ the sigma-algebra generated by the arrival times up to time $t$, as in Observation \ref{obs:coupling}, and by $\bbE_{\cF_t}$ the corresponding expectation. Then for any $\o'\in\O^{(g)}$
\begin{equation}
\label{eq:roto'}
\r^\o_t(\o')=\bbE_{\cF_t}\left[\bbP^\o(\o(t)=\o'\tc \cF_t)\right]=\bbE_{\cF_t}\left[\prod_{x\not\in\X_t}\1_{\o'_x=\o_x}\prod_{x\in\X_t}\r_x(\o'_x\tc S_{(\f(\o(t)))_x})\right],
\end{equation}
the last equality reflecting that by Observation \ref{obs:coupling} $\X_t$ and $\f(\o(t))$ are $\cF_t$-measurable. Again by Observation \ref{obs:coupling}, $\X_t$ and $\f(\o(t))$ are the same for all $\o$ in the support of $\n^\h$, so we denote the latter by $\h(t)$. If we now average \eqref{eq:roto'} over the initial condition $\o$ w.r.t.\ $\n^\h$ and use \eqref{eq:nueta}, we obtain
$\r^{\n^\h}(\o')=\bbE_{\cF_t}[\n^{\h(t)}(\o')],$
which is exactly \eqref{eq:rtnh}, since $\h(t)$ has the law $\m_t^\eta$ of CBSEP with initial state $\h$, as it is the projection of {$g$-{CBSEP}} with initial condition $\o$ such that $\f(\o)=\h$
\end{proof} 
Next we write
\begin{align}
\label{eq:Fabio}
\max_{\o\in \O^{(g)}_+}\|\rho_t^\o-\rho_+\|_{\rm TV}{}\le
  \max_{\eta\in \O_+}\left(
  \max_{\o:\f(\o)=\eta}\left\|\rho_t^\o-\rho_t^{\nu^\eta}\right\|_{\rm TV}
  +\left\|\rho_t^{\nu^\eta}-\rho_+\right\|_{\rm TV}\right),
  \end{align}
  where $\rho_+$ was defined in Section \ref{subsec:CBSEP}.
Using Claim \ref{claim:law}, it follows that
\begin{equation}
  \label{eq:10}
\max_{\eta\in \O_+}\left\|\rho_t^{\nu^\eta}-\rho_+\right\|_{\rm TV} =
\max_{\eta\in \O_+}
\|\mu_t^\eta-\mu\|_{\rm TV},
\end{equation}
where $\mu$ is the reversible measure of {CBSEP} with parameter $p$.  

To bound the first term in the r.h.s. of \eqref{eq:Fabio} the key ingredient is to use the graphical construction to embed into {$g$-{CBSEP}} of a suitable continuous time simple random walk $(W_t)_{t\ge 0}$ on $G$ 
with the property that $g$-CBSEP at time $t$ has a "particle" at the location of $W_t$. 

Given $\o\in\O^{(g)}_+$, let $v\in V$ be such that $\f(\o_v)=1,$ and let $t^*=\min \{t_n^{(u,v)}> 0\}$ be
the first time an edge of the form $(u,v)$ is resampled to produce a configuration $\o'$ with $\o'_u\in S_1$ and $\o'_v\in S_0$. We then set $W_s=v$ for $s<t^*$ and $W_{t^*}=u$. By iterating the construction we construct $(W_t)_{t\ge 0}$ with $W_0=v$. It is clear that $\f(\o_{W_t(\o)}(t))=1$ for all $t$ and that the law $\bbP_v(\cdot)$ of $(W_t)_{t\ge 0}$ is that of a continuous-time random walk started at $v$ and jumping to a uniformly chosen neighbour at rate $d_{W_t}(1-p)/(2-p)$. 
We denote by $\scov$ the cover time of $(W_t)_{t\ge 0}$.\footnote{We use $\scov$ to distinguish it from the cover time $\t_{\rm cov}$ of the \emph{discrete time} simple random walk on $G$.} 

Observation \ref{obs:coupling} then implies
  \begin{equation}
    \label{eq:11}
\max_{\eta\in \O_+}
  \max_{\o:\f(\o)=\eta}\left\|\rho_t^\o-\rho_t^{\nu^\eta}\right\|_{\rm TV}\le \max_{\substack{\o,\o'\in \O_+^{(g)}\\\f(\o)=\f(\o')} }\|\rho_t^\o-\rho_t^{\o'}\|_{\rm TV}
\le \max_{v\in V} \bbP_v(\scov>t).
  \end{equation}
The upper bound given in the theorem now follows immediately from \eqref{eq:Fabio}, \eqref{eq:10}, and \eqref{eq:11} together with a standard
comparison between $\scov$ and the cover time of the \emph{discrete time} simple random walk on $G$.

\begin{appendix}
\section*{Proof of Lemma \ref{lem:birth-death}}
\label{appn}
Recall that
\[
  \g(k)={n\choose k} \frac{p^k}{(1-p)^{k}}\frac{(1-p)^n}{1-(1-p)^n}
\]
and consider the birth and death process on $\{1,\dots, n\}$
reversible w.r.t.\ $\g$ with Dirichlet form
\begin{equation*}
  \cD_\g(g)=\sum_{k=2}^n \g(k) k [g(k)-g(k-1)]^2,
       \end{equation*}
corresponding to the jump rates $c(1,0)=0$ and 
\begin{align*}
c(k,k-1)&{}=k &k&{}=2,\dots, n\\
c(k,k+1)&{}=(n-k)\frac{p}{1-p}&k{}&=1,\dots,n-1,
\end{align*}
Let $m=\lceil pn \rceil$ and  $i=\max(2,m)$. Using \cite{Miclo99}*{Proposition 4} (see also \cite{Yang08}) the logarithmic Sobolev constant of the above chain
is bounded from above, up to an absolute multiplicative constant, by 
the number $ C_*=C_-\vee C_+$,
where
\begin{align}
  C_+&=\max_{j\ge i+1}\left(\sum_{k=i+1}^{j}\frac{1}{\g(k)c(k,k-1)}\right)\g(N\ge
          j)|\log\left(\g(N\ge j)\right)|,\nonumber\\
  C_-&=\max_{j\le i-1}\left(\sum_{k=j}^{i-1}\frac{1}{\g(k)c(k,k+1)}\right)\g(N\le
          j)|\log\left(\g(N\le j)\right)|.\label{eq:def:C-}
\end{align}
Assume first that $i=m$ and let us start with $C_+$. For $\ell\ge 1$ write $a_\ell=\frac{1}{(m+\ell)\g(m+\ell)}$ and $S_k=\sum_{\ell=1}^k a_\ell$. We have
\[
  \frac{a_{\ell+1}}{a_{\ell}}=\frac{1-p}{p}\frac{m+\ell}{n-m-\ell}\ge
  1,
\]
from which it follows that for $0<\d<1$ we have
\begin{align}
\label{eq:ratio:bin}
\frac{a_{\ell+1}}{a_{\ell}}&= 1+\Theta(\ell/m)=e^{\Theta(\ell/m)} &\ell&{}\le m,\nonumber\\
\frac{a_{\ell+1}}{a_{\ell}}&\ge \frac{(1-p)(m+\d m)}{p(n-m)}\ge 1+\d&\ell&{}\ge \d m.  
\end{align}
In particular, for any two integers $s\le t\le m$ such that $t-s\ge \min(\sqrt{m},m/s)$, 
it holds that for some absolute constant $\b>1$
\begin{equation}
\label{eq:ratio:partial}
\frac{a_{t}}{a_s}=\prod_{\ell=s}^{t-1}\frac{a_{\ell+1}}{a_{\ell}}
=e^{\Theta((t-s)t/m)}
\ge \b.
\end{equation}
We first analyse the behaviour of $S_k\g(N\ge m+k)|\log(\g(N\ge m+k))|$ for $k\le \d m$ where $\d>0$ is
a sufficiently small constant depending on $\b$. 
\begin{lem}
\label{lem:app}
There exists a constant $c>0$ such that for $\d>0$ small enough and $k\le \d m$ we have
  \[
    S_k\g(N\ge m+k)|\log(\g(N\ge m+k))|\le c
    \]
  \end{lem}
  \begin{proof}
Let $0<\d<1$. Define recursively
\[
  k_0=1, \quad k_1=\lceil\sqrt{m}\,\rceil, \quad k_{t+1}=k_t+\lceil
  m/k_{t}\rceil,
\]
and let $T$ be the first index such that $k_T\ge \d m$.
Using
\eqref{eq:ratio:partial} together with $a_{\ell+1}\ge a_\ell,$
$k_{t+1}-k_t\le k_t-k_{t-1},$ and $k_t/m\le \d,$ we claim
that for any $2\le t\le T-1$
\begin{equation}
\begin{aligned}
  \frac{(S_{k_{t+1}}-S_{k_t})}{(S_{k_t}-S_{k_{t-1}})}={}&
  \frac{\sum_{\ell=k_t+1}^{k_{t+1}}a_\ell}{\sum_{\ell=k_{t-1}+1}^{k_{t}}a_\ell}\ge \b \frac{k_{t+1}-k_t}{k_{t}-k_{t-1}}\\
  \ge{}& \b\left(\frac{k_t}{k_{t-1}}+\frac{k_t}{m}\right)^{-1}\ge \b\left(\frac{k_t}{k_{t-1}}+\d\right)^{-1}. 
  \end{aligned}
  \label{eq:7}
\end{equation}
To prove the first inequality in \eqref{eq:7}, observe that for any positive non-decreasing sequence
$(a_j)_{j=1}^\infty$ and positive integers $m\le n$,
\begin{align*}
\frac{a_{n+1}+\dots+a_{n+m}}{a_1+\dots +a_n}
  \ge{}& \min_{j}\left(\frac{a_{j+m}}{a_{j}}\right)
  \left(\frac{a_{n-m+1}+\dots+a_n}{a_1+\dots +a_n}\right)\\
\ge{}& \min_{j}\left(\frac{a_{j+m}}{a_{j}}\right)
  \left(\frac{\sum_{j=n-m+1}^n a_j}{(n-m)a_{n-m}+\sum_{j=n-m+1}^n a_j}\right)\\
  \ge{}& \min_{j}\left(\frac{a_{j+m}}{a_{j}}\right) \frac mn,
\end{align*}
because $\sum_{j=n-m+1}^n a_j\ge a_{n-m}m$.

If now $\d,t$ are chosen
small enough and large enough, respectively, depending on the constant $\b$ above, the
r.h.s.\ of \eqref{eq:7} is greater than e.g.\ $\b^{1/2}>1$.  In other words, fixing $\d$ small enough and $t_0$ large enough, the sequence $\left(\left(S_{k_{t+1}}-S_{k_t}\right)\right)_{t=t_0}^T, t_0\gg 1,$ is exponentially increasing. 

Now fix $k\le\d m$ and $t$ such that $k_t\le k< k_{t+1}$. Assume first that $t_0\le t<T$. Then, for some positive constant $c$ allowed to depend on $\b$ and $t_0$ and to change from line to line, we have
\begin{equation*}
    \label{eq:Sk}
    \begin{aligned}
S_k&{}\le \sum_{s=t_0}^{t+1} \left(S_{k_s}-S_{k_{s-1}}\right) +
     S_{t_0}\le c \left(S_{k_{t+1}}-S_{k_{t}}\right)+S_{t_0}\\
&{}\le c\frac{k_{t+1}-k_t}{m\g(m+k_{t+1})}\le c\frac{k_{t+1}-k_t}{m\g(m+k)}. 
\end{aligned}
\end{equation*}
If instead $0\le t<t_0$, we directly have that 
\[S_k\le ka_k\le c\frac{k_{t+1}-k_t}{m\g(m+k)}.\]
Using the bounds
\begin{align*}
  \g(N\ge m+k)&{}\le c\frac{m+k}{k}\g(m+k),&|\log\left(\g(N\ge m+k)\right)|&{}\le c k^2/m,
\end{align*}
we finally get that $k_t\le k< k_{t+1}, t<T,$
\[S_k\g(N\ge m+k)|\log(\g(N\ge m+k))|\le
  c\frac{(k_{t+1}-k_t)k_{t+1}}{m}\le c.\qedhere\]
\end{proof}

Let $\d$ be as in Lemma \ref{lem:app}. We next consider the easier case, $k\ge \d m$. By
\eqref{eq:ratio:bin}, for $c$ large enough depending on $\d$ and allowed to change from line to line, we have that $S_k\le 
S_{k_T}+ca_k\le ca_k$ and
$\g(N\ge m+k)\le c\g(m+k)$. Thus, for $k\ge \d m$, we have that
\[
  S_k\g(N\ge m+k)|\log(\g(N\ge
  m+k))| \le \frac{c}{m+k}|\log(\g(N\ge
    m+k))|\le c\log(1/p),\]
since for all $k$ we trivially have $\g(m+k)\ge p^{m+k}$. In conclusion, we have proved that
$C_+\le O(\log(1/ p))$ if $m\ge 2$. If instead $m=1$, then the very same computations still give $C_+\le O(\log(1/
p))$, Lemma \ref{lem:app} being void.

The bound of $C_-$ follows the same pattern. If $m=O(1)$, the reader
may readily check that $C_-=O(1)$ because all terms in \eqref{eq:def:C-} are
$O(1)$. If instead $m\gg 1$, we still obtain $C_-=O(1)$, concluding the proof of Lemma \ref{lem:birth-death}.
\end{appendix}

\section*{Acknowledgements}
We wish to thank the Department of Mathematics and Physics of
University Roma Tre for the kind hospitality. Enlightening discussions
with P. Caputo, Y. Peres, J. Salez, A. Shapira, J. Swart and P. Tetali are also warmly acknowledged. We thank the anonymous referees for their helpful remarks, which helped improve the presentation of the paper.\\
We acknowledge financial support from: ERC Starting Grant 680275 ``MALIG'', ANR-15-CE40-0020-01 and PRIN 20155PAWZB ``Large Scale Random Structures''.
\bibliographystyle{plain}
\bibliography{Bib}
\end{document}